%% file: main.tex
\documentclass[12pt]{article}
\usepackage[utf8]{inputenc}
\usepackage{mypreamble}

\title{The Hopf Monoid of Orbit Polytopes}
\author{Mariel Supina\footnote{The author was supported by the Chancellor's Fellowship of the University of California, Berkeley and the Graduate Fellowships for STEM Diversity.}}
\date{}

\begin{document}

\maketitle

\abstract{Many families of combinatorial objects have a Hopf monoid structure.
Aguiar and Ardila introduced the Hopf monoid of generalized permutahedra and showed that it contains various other notable combinatorial families as Hopf submonoids, including graphs, posets, and matroids.
We introduce the Hopf monoid of orbit polytopes, which is generated by the generalized permutahedra that are invariant under the action of the symmetric group.
We show that modulo normal equivalence, these polytopes are in bijection with integer compositions.
We interpret the Hopf structure through this lens, and we show that applying the first Fock functor to this Hopf monoid gives a Hopf algebra of compositions.
We describe the character group of the Hopf monoid of orbit polytopes in terms of noncommutative symmetric functions, and we give a combinatorial interpretation of the basic character and its polynomial invariant.}

\section{Introduction}

In \cite{AA}, Aguiar and Ardila introduced the Hopf monoid of generalized permutahedra and proved that much of its algebraic structure can be interpreted combinatorially.
Many other combinatorial families form Hopf submonoids of generalized permutahedra, so this theory produced new proofs of known results about graphs, matroids, posets, and other objects.
It also led to some new and surprising theorems.
Associated to a Hopf monoid is a group of multiplicative functions called the character group.
Aguiar and Ardila showed that the character groups of the Hopf monoids of permutahedra and associahedra are isomorphic to the groups of formal power series under multiplication and composition, respectively.
Using their formula for the antipode of generalized permutahedra, they found that permutahedra have information about multiplicative inverses of power series encoded in their face structure, and associahedra have analogous information about compositional inverses of power series.

A subject of ongoing study is to examine other Hopf submonoids of generalized permutahedra (for example, see \cite{BBM}).
We consider the Hopf monoid generated by orbit polytopes.
These are the generalized permutahedra that are invariant under the action of the symmetric group, so this Hopf monoid contains permutahedra as a Hopf submonoid.

This paper presents two main results.
First, Theorem \ref{thm:hopfalgebra} describes a Hopf algebra of compositions which results from applying the first Fock functor to the Hopf monoid of orbit polytopes.
Second, Theorem \ref{prop:mainresult} shows that the character group of the Hopf monoid of orbit polytopes is isomorphic to a subgroup of the group of invertible elements in the completion of the Hopf algebra $NSym$ of noncommutative symmetric functions.
In Section \ref{sec:prelim}, we introduce some necessary background.
Section \ref{sec:ops} formally defines orbit polytopes and shows how, up to normal equivalence, they can be viewed as compositions.
In Section \ref{sec:algebraic}, we interpret the product and coproduct of generalized permutahedra in the case of orbit polytopes, and we show that these operations can be neatly described in terms of compositions.
Finally, in addition to presenting the description of the character group, Section \ref{sec:character} discusses the basic character of orbit polytopes and its polynomial invariant.

\section{Preliminaries}\label{sec:prelim}

Let $I$ be a finite set. 
Let $\reals I$ be the linearization of $I$, so $\reals I$ is a real vector space with basis $\{e_i:i\in I\}$.
Let $\reals^I$ be its dual, the set of linear functionals $y:\reals I\to\reals$.

\subsection{Normal Equivalence of Polytopes}

We begin by introducing an important equivalence relation on polytopes which will be very useful in our investigation of orbit polytopes.
Let $P\subseteq \reals I$ be a polytope and let $F$ be a face of $P$ (we write $F\le P$).

\begin{definition}
The \textit{normal cone} of $F$ is the cone of linear functionals
	\[\mathcal{N}_P(F) := \{y\in\reals^I:y \text{ attains its max value on } P \text{ at every point in } F\},\]
i.e. $\mathcal{N}_P(F)$ is the cone of linear functionals that define a face of $P$ containing $F$.
\end{definition}

\begin{definition}
The \textit{normal fan} $\mathcal{N}_P$ of $P$ is the fan in $\reals^I$ consisting of the normal cones of each face of $P$,
    \[\mathcal{N}_P:= \{\mathcal{N}_P(F):F\le P\}.\]
\end{definition}
	
\begin{definition}
The polytopes $P$ and $Q$ are \textit{normally equivalent} if $\mathcal{N}_P=\mathcal{N}_Q$.
\end{definition}

\begin{example}
The polytopes $P$ and $Q$ in Figure \ref{fig:normeq} are normally equivalent.
\end{example}

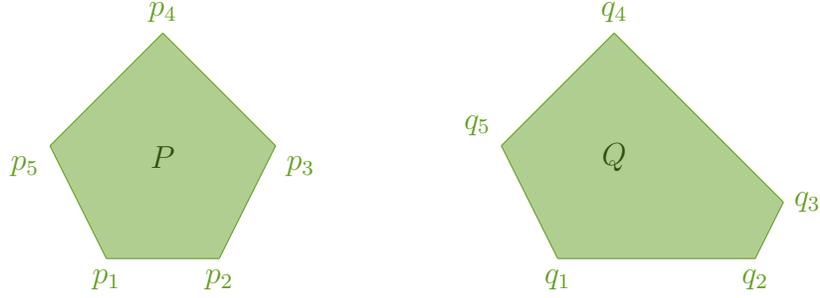
\begin{figure}
\begin{center}
\input{img/normally_equivalent_polytopes.tikz}
\caption{Normally equivalent polytopes in $\reals^2$.}
\label{fig:normeq}
\end{center}
\end{figure}

\subsection{The Braid Arrangement}\label{sec:braid}

In this section, we introduce an important hyperplane arrangement that is central to the study of generalized permutahedra.

\begin{definition}
The \textit{braid arrangement} is the hyperplane arrangement in $\reals I$ consisting of the hyperplanes $x_i=x_j$ for $i,j\in I$ with $i\ne j$.
\end{definition}
	
The braid arrangement divides $\reals I$ into closed full-dimensional cones, or \textit{chambers}.

\begin{definition}
The \textit{braid fan} $\braid$ is the polyhedral fan in $\reals I$ formed by taking the set of chambers of the braid arrangement and all of their faces.
\end{definition}

We will also sometimes deal with the \textit{dual braid fan} $\dualbraid$ that lives in the dual space $\reals^I$.
It is defined analogously to $\braid$ and shares all of its combinatorial properties.
The following is a well-known fact about the braid fan:

\begin{proposition}\label{prop:osp}
The faces of the braid fan are in bijection with ordered set partitions of $I$.
\end{proposition}

Our convention in this paper will be that the ordered set partition $S_1\sqcup\dots\sqcup S_k$ of $I$ corresponds to the cone of the braid arrangement consisting of all points satisfying, for $i_1\in S_j$ and $i_2\in S_{\ell}$:
    \begin{enumerate}[(i)]
    \item $x_{i_1}\ge x_{i_2}$ if and only if $j\le\ell$, and
    \item $j=\ell$ implies $x_{i_1}=x_{i_2}$.
    \end{enumerate}
It follows that the chambers of the braid fan are in bijection with the set of all linear orderings of the coordinates of a point in the ambient space.
We will need to distinguish one chamber to use as a reference point.

\begin{definition}\label{def:fundom}
Fix any chamber $\fundom$ of the braid fan, and call this the \textit{fundamental chamber}.
Define a linear order $\prec$ on $I$ to be the one corresponding to $\fundom$, so
    \[\fundom = \{x\in\reals I:i\prec j\Longrightarrow x_i\ge x_j\}.\]
\end{definition}

For example, when $I=[n]$ we could choose the fundamental chamber of $\mathcal{B}_{[n]}$ to be the set of points in $\reals^n$ with $x_1\ge x_2\ge\dots\ge x_n$.

\begin{corollary}\label{cor:fundamental_chamber}
The faces of the fundamental chamber $\fundom$ of the braid fan $\braid$ are in bijection with the compositions of the integer $|I|$.
\end{corollary}
\begin{proof}
 From Definition \ref{def:fundom} we can see that the faces of $\fundom$ correspond only to the ordered set partitions $S_1\sqcup\dots\sqcup S_k$ of $I$ such that for $1\le i<k$, all elements of $S_i$ are greater than all elements of $S_{i+1}$ under $\prec$.
Thus the ordered set partitions corresponding to faces of $\fundom$ depend only on the sizes of their parts.
This sequence of sizes $(|S_1|,\dots,|S_k|)$ is an integer composition of $|I|$.
\end{proof}

\subsection{Generalized Permutahedra}

Generalized permutahedra are polytopes with very nice combinatorial and algebraic properties.
They are equivalent to submodular functions and to polymatroids up to translation.

\begin{definition}
Let $\mathcal{F}$ and $\mathcal{G}$ be polyhedral fans.
Then $\mathcal{F}$ is a \textit{coarsening} of $\mathcal{G}$ if each cone of $\mathcal{G}$ is contained in a cone of $\mathcal{F}$.
\end{definition}

\begin{definition}
A \textit{generalized permutahedron} is a polytope whose normal fan is a coarsening of the braid fan.
\end{definition}

Generalized permutahedra can be obtained by moving the vertices of a standard permutahedron in such a way that the edge directions are preserved \cite{Postnikov}.
Proposition \ref{prop:gpsf} below gives an equivalent definition of generalized permutahedra using submodular functions.

\begin{definition}
A \textit{submodular function} is $z:2^I\to \reals$ where $I$ is a finite set and $z$ satisfies the following properties:
	\begin{enumerate}[(i)]
	\item $z(\varnothing) = 0$, and
	\item $z(S\cap T)+z(S\cup T)\le z(S)+z(T)$ for all $S,T\subseteq I$.
	\end{enumerate}
\end{definition}

\begin{definition}
Let $z:2^I\to\reals$ be a submodular function.
The \textit{base polytope} of $z$ is 
	\begin{equation}\label{eq:basepolytope}
	\mathcal{P}(z) = \Big\{ x\in\reals I:\sum_{i\in I}x_i = z(I)\text{ and }\forall\; \varnothing\subsetneq A\subsetneq I,\;\sum_{a\in A}x_a \le z(A) \Big\}.
	\end{equation}
\end{definition}

\begin{proposition}[\cite{Fujishige}]\label{prop:gpsf}
A polytope is a generalized permutahedron if and only if it is the base polytope of a submodular function.
Furthermore, every generalized permutahedron is the base polytope a unique submodular function, and all of the inequalities in (\ref{eq:basepolytope}) are tight.
\end{proposition}

\section{Orbit Polytopes}\label{sec:ops}

In this section, we introduce orbit polytopes, the main combinatorial objects studied in this paper.

\subsection{Definition of an Orbit Polytope}

The symmetric group $\sym_I$ acts on $\reals I$ by permuting coordinates.
If $\sigma$ is a permutation in $\sym_I$ and $p=(p_i)_{i\in I}\in\reals I$, then this action is given by $\sigma(p)_i=p_{\sigma^{-1}(i)}$.

\begin{definition}\label{def:op}
Let $p\in\reals I$.
The \textit{orbit polytope} of $p$ is the polytope
	\[\O(p):=\conv\{\sigma(p):\sigma\in \sym_I\}.\]
\end{definition}

Orbit polytopes are sometimes called permutahedra \cite{Postnikov}, but we avoid this terminology in order to distinguish the Hopf monoid of orbit polytopes from the Hopf monoid of (standard) permutahedra discussed in \cite{AA}.

Orbit polytopes are closely related to \textit{weight polytopes}, a general construction arising in representation theory and the theory of finite reflection groups.
The vertices of weight polytopes are given by the orbit of a special point, called a weight, under a relevant action.
The weights arising in the representation theory of the general linear group are all integer points; thus orbit polytopes with integer vertices are the same as weight polytopes for $GL_n$.

\begin{example}
The orbit polytope $\O(1,0,0)$ is the $2$-dimensional standard simplex in $\reals^3$.
The orbit polytope $\O(1,1,0)$ is combinatorially equivalent to $\O(1,0,0)$, but it has a different normal fan (see Figure \ref{fig:simplex}).
    \begin{figure}
    \begin{center}
    \input{img/simplex.tikz}
    \caption{The orbit polytopes $\O(1,0,0)$ and $\O(1,1,0)$.}
    \label{fig:simplex}
    \end{center}
	\end{figure}
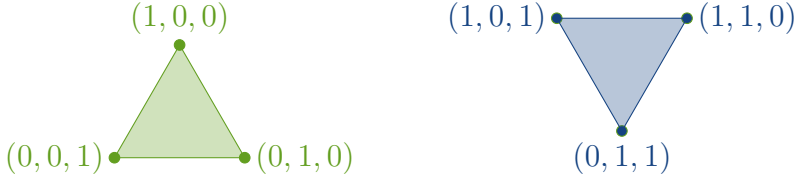
\end{example}
	
\begin{example}
For any $\lambda\in \reals$, the orbit polytope $\O(\lambda,\dots,\lambda)$ is a single point in $\reals I$.
\end{example}
	
\begin{example}
Let $I=[n]$ and $p=(n,n-1,\dots,1)\in\reals^n$.
Then $\O(p)$ is the standard $n$-permutahedron (see Figure \ref{fig:permutahedron}).
    \begin{figure}
    \begin{center}
    \input{img/permutahedron.tikz}
    \caption{The standard $4$-permutahedron $\O(4,3,2,1)$.}
    \label{fig:permutahedron}
    \end{center}
	\end{figure}
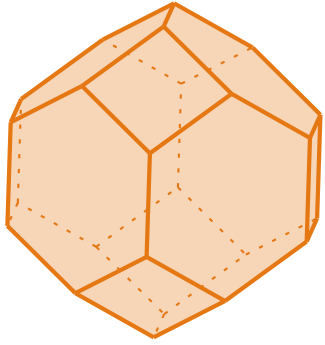
\end{example}

Note that if $q$ is an element of the $\sym_I$ orbit of $p$, then $\O(p)=\O(q)$.
In this paper, we will often use this observation to write $p$ with its coordinates in decreasing order.
Furthermore, note that the orbit polytope $\O(p)$ for $p=(p_i)_{i\in I}\in\reals I$ will always lie in the codimension-$1$ affine subspace of $\reals I$ given by
    \[\sum_{i\in I}x_i=\sum_{i\in I}p_i.\]

To describe the faces of an orbit polytope, it is helpful to use the \textit{rearrangement inequality}.
This states that if $x_1,\dots,x_n$ and $y_1,\dots,y_n$ are real numbers such that $x_1\ge x_2\ge\dots\ge x_n$ and $y_1\ge y_2\ge\dots\ge y_n$, then for all $\sigma\in \sym_n$, we have
    \begin{equation}\label{ineq:rearrangement_inequality}
    x_1y_1+\dots+x_ny_n\ge x_{\sigma(1)}y_1+\dots+x_{\sigma(n)}y_n.
    \end{equation}
We can see inequality (\ref{ineq:rearrangement_inequality}) in action in the next example.

\begin{example}
Let $p=(2,2,1,0)\in\reals[4]$, and let $y=(1,0,1,1)\in\reals^{[4]}$.
Then the rearrangement inequality says that the vertices of $\O(p)$ maximizing $y$ are the elements of the $\sym_4$ orbit of $p$ such that the three largest coordinates are in positions $1$, $2$, and $4$, and the smallest coordinate is in position $2$.
These are the points $(1,0,2,2)$, $(2,0,1,2)$, and $(2,0,2,1)$.
Hence the $y$-maximal face of $\O(p)$ is the triangular facet $\conv\{(1,0,2,2), (2,0,1,2),(2,0,2,1)\}$, as seen in Figure \ref{fig:maximal_face}.
    \begin{figure}
    \begin{center}
    \input{img/maximal_face.tikz}
    \caption{The $(1,0,1,1)$-maximal face of $\O(2,2,1,0)$.}
    \label{fig:maximal_face}
    \end{center}
	\end{figure}
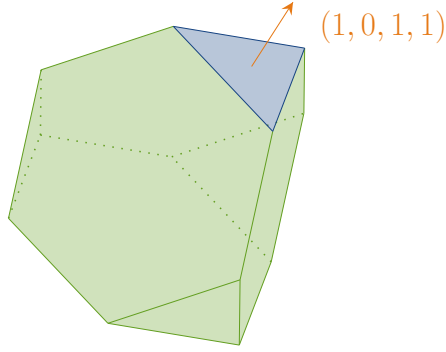
\end{example}

\begin{proposition}\label{prop:gp}
Orbit polytopes are generalized permutahedra.
\end{proposition}
\begin{proof}
Let $p\in\reals I$ for some finite set $I$, and consider the orbit polytope $\O(p)$.
Choose some linear functional $y\in\reals^I$ and let $F$ be the $y$-maximal face of $\O(p)$.
Now, $y$ lies in the relative interior of exactly one cone $\mathcal{C}\in\dualbraid$.
This cone $\mathcal{C}$ corresponds to some ordered partition $S_1\sqcup S_2\sqcup\dots\sqcup S_k$ of the set $I$.
Inequality (\ref{ineq:rearrangement_inequality}) says that the vertices of $F$ must be precisely the elements of the $\sym_I$ orbit of $p$ that have their largest $|S_1|$ coordinates in the positions in $S_1$, the next largest $|S_2|$ coordinates in the positions in $S_2$, and so on, ending with the smallest $|S_k|$ coordinates, which must be in the positions in $S_k$.
But then by (\ref{ineq:rearrangement_inequality}), all other $z\in\mathcal{C}$ must also attain their maximum value for $\O(p)$ on $F$.
Thus $\mathcal{C}\subseteq\mathcal{N}_{\O(p)}(F)$, so $\mathcal{N}_{\O(p)}$ coarsens $\dualbraid$.
\end{proof}

\subsection{Half-Space Description of Orbit Polytopes} \label{sec:hspace}

Polytopes can be described either as the convex hull of their vertices or as the intersection of the half-spaces defining their facets.
Orbit polytopes are defined in terms of their vertex description (see Definition \ref{def:op}), but they have a simple half-space description as well.

\begin{proposition}[\cite{Rado}]\label{prop:halfspace_desc}
Let $p\in\reals I$ with $n=|I|$, and label the coordinates of $p$ as $p_1,\dots,p_n$ such that $p_1\ge\dots\ge p_n$.
Then
    \begin{equation}\label{eq:halfspace_desc}
    \O(p) = \bigg\{x\in\reals I: \sum_{i\in I}x_i = \sum_{i=1}^np_i\text{ and }\forall\, \varnothing\subsetneq S\subsetneq I,\;\sum_{s\in S}x_s\le\sum_{i=1}^{|S|}p_i\bigg\}.
    \end{equation}
\end{proposition}
\begin{proof}
By Proposition \ref{prop:gp}, $\O(p)$ is a generalized permutahedron.
Thus by Proposition \ref{prop:gpsf}, $\O(p)$ is the base polytope of a unique submodular function.
Consider the function $z:2^I\to\reals$ given by $z(S)=\sum_{i=1}^{|S|}p_i$.
It is straightforward to verify that $z$ is submodular, and that the right side of equation (\ref{eq:halfspace_desc}) is the base polytope of $z$.
Moreover, it is immediate that all vertices of $\O(p)$ satisfy the equation and inequalities in (\ref{eq:halfspace_desc}).
Thus by convexity, all points in $\O(p)$ satisfy these inequalities.
For each $S\subseteq I$ there exists some vertex $v$ of $\O(p)$ satisfying
    \[\sum_{s\in S}v_s = \sum_{i=1}^{|S|}p_i\;;\]
namely, $v$ is any element of the $\sym_I$ orbit of $p$ with $p_1,\dots,p_{|S|}$ in the positions corresponding to $S$.
Hence, these inequalities are tight for $\O(p)$, so $\O(p)$ is the base polytope of $z$.
\end{proof}

\subsection{Orbit Polytopes and Submodular Functions}

It is clear from Definition \ref{def:op} that orbit polytopes are invariant under the action of $\sym_I$.
However, one might wonder whether there exist other generalized permutahedra that are invariant under this action.
Submodular functions are a useful tool for answering this question.
Proposition \ref{prop:halfspace_desc} shows that $\O(p)=\mathcal{P}(z)$ where $z:2^I\to\reals$ is the submodular function given by
    \begin{equation}\label{eq:submodular}
    z(S) = \sum_{i=1}^{|S|}p_i
    \end{equation}
for all $\varnothing\subsetneq S\subseteq I$.

In general, the action of $\sym_I$ on $I$ induces an action of $\sym_I$ on submodular functions on $I$ by $(\sigma\cdot z)(S) = z(\sigma^{-1}(S))$.
One can show that under this action, a submodular function $z:2^I\to\reals$ is $\sym_I$-\textit{invariant} if $z(S)=z(T)$ whenever $|S|=|T|$.

\begin{proposition}
A polytope in $\reals I$ is an orbit polytope if and only if it is the base polytope of an $\sym_I$-invariant submodular function.
\end{proposition}
\begin{proof}
It is clear that the submodular function in (\ref{eq:submodular}) is $\sym_I$-invariant.
Conversely, let $|I|=n$ and let $z:2^I\to\reals$ be $\sym_I$-invariant.
Then there exist $t_1,\dots,t_n\in\reals$ such that for all $\varnothing \subsetneq S\subseteq I$, we have $z(S)=t_{|S|}$.
Let $t_0=0$ and define $p_k=t_k-t_{k-1}$ for $1\le k\le n$.
Label the elements of $I$ as $1,2,\dots,n$.
Then by submodularity, for $2\le k\le n$ we have
    \begin{align*}
        z(\{1,\dots,k-2\}) +z(\{1,\dots,k\})&\le z(\{1,\dots,k-1\}) + z(\{1,\dots,k-2,k\})\\
        t_{k-2}+t_k&\le t_{k-1}+t_{k-1}\\
        t_k-t_{k-1}&\le t_{k-1}-t_{k-2}\\
        p_k&\le p_{k-1}.
    \end{align*}
Thus $p_1\ge\dots\ge p_n$.
It follows from Proposition \ref{prop:halfspace_desc} that $\mathcal{P}(z)=\O(p)$ where the multiset of coordinates of $p$ is $\{p_1,\dots,p_n\}$.
\end{proof}

\begin{corollary}
Orbit polytopes are exactly the generalized permutahedra which are invariant under the $\sym_I$ action on $\reals I$.
\end{corollary}

The set of all submodular functions $z:2^I\to\reals$ forms the \textit{submodular cone} in $\reals^{(2^I)}$.
The structure of the submodular cone is extremely complicated and a subject of ongoing study.
The cone of submodular functions corresponding to orbit polytopes admits a much simpler description.
We can obtain this cone from the submodular cone by intersecting it with all hyperplanes of the form $z(S)=z(T)$ for pairs $S,T\subseteq I$ with $|S|=|T|$.
The faces of the submodular cone correspond to normal equivalence classes of generalized permutahedra, with inclusion of faces corresponding to refinement of normal fans.
This means that the faces of our new cone correspond to normal equivalence classes of orbit polytopes.
We will see in Section \ref{sec:opnormaleq} that these equivalence classes are in bijection with integer compositions, which implies that the face lattice of the cone of $\sym_I$-invariant submodular functions is simply a Boolean lattice.
Thus the cone of $\sym_I$-invariant submodular functions is simplicial.

\subsection{Orbit Polytopes Modulo Normal Equivalence}\label{sec:opnormaleq}

The goal of this section is to show that normal equivalence classes of orbit polytopes $\O(p)$ for $p\in\reals I$ are in bijection with compositions of the integer $n:=|I|$.
In order to show this, we will need to use some nice properties of orbit polytopes.

\begin{lemma}\label{lem:onevertex}
Orbit polytopes have exactly one vertex lying in each chamber of the braid fan $\braid$.
\end{lemma}
\begin{proof}
It is a general fact in the theory of Coxeter groups that a finite reflection group acts transitively on the chambers of its Coxeter arrangement (see \cite[1.12]{Humphreys}).
Thus the $\sym_I$-orbit of a point in $\reals I$ must consist of exactly one point in each chamber of $\braid$.
\end{proof}

Note that it is possible for a vertex to lie in more than one chamber of $\braid$ if it is on the intersection of two or more closed chambers.
We can identify each cone of the braid fan $\braid$ with the cone of the dual braid fan $\dualbraid$ that corresponds to the same ordered set partition.
Some would consider it a mathematical sin to conflate the dual vector spaces $\reals I$ and $\reals^I$ in this way.
However, in the case of orbit polytopes, such sinfulness can be illuminating, as evidenced by the following lemma.

\begin{lemma}\label{lem:vertexcone}
Let $v$ be a vertex of the orbit polytope $\O\subset\reals I$.
Then the normal cone $\mathcal{N}_{\O}(v)$ is the union of all the chambers of $\braid$ containing $v$, viewed as cones in $\reals^I$ under the natural correspondence between $\braid$ and $\dualbraid$.
\end{lemma}
\begin{proof}
Suppose $\O=\O(p)$ for $p\in\reals I$.
By Lemma \ref{lem:onevertex}, $\O$ has exactly one vertex in each chamber of $\braid$, and this vertex results from reordering the coordinates of $p$ according to the ordered set partition of that chamber.
The rearrangement inequality (\ref{ineq:rearrangement_inequality}) implies that the functionals in the normal cone of a vertex $v$ must have their coordinates ordered in the same way as $v$.
Thus the functionals in $\mathcal{N}_{\O}(v)$ are those in all chambers of $\dualbraid$ corresponding to chambers of $\braid$ containing $v$.
\end{proof}

Given a point $p\in\reals I$, we can obtain a composition of $n$ using Corollary \ref{cor:fundamental_chamber} and Lemma \ref{lem:onevertex} as follows:

\begin{definition}
The \textit{composition of} $p$ is the integer composition corresponding to the unique face of the fundamental chamber $\fundom$ of $\braid$ that contains some point in the $\sym_I$-orbit of $p$.
\end{definition}

\begin{example}
Let $p=(1,3,1,6,6,0,2,1)\in\reals^8$.
Then the element of the $\sym_8$-orbit of $p$ that lies in the fundamental chamber of $\mathcal{B}_{[8]}$ is $p'=(6,6,3,2,1,1,1,0)$.
The composition of the integer $8$ corresponding to $p$ is $(2,1,1,3,1)$ since $p'$ has two of its biggest coordinate ($6$), one of the second biggest coordinate ($3$), one of the third biggest coordinate ($2$), three of the fourth biggest coordinate ($1$), and one of the smallest coordinate ($0$).
\end{example}

\begin{corollary}\label{cor:vertexcomp}
Every vertex of an orbit polytope $\O$ has the same composition.
\end{corollary}
\begin{proof}
This follows from Lemma \ref{lem:onevertex}, which implies that $\O$ has exactly one vertex in the fundamental chamber.
\end{proof}

\begin{definition}
The \textit{composition of an orbit polytope} $\O$ is the composition of any of its vertices.
\end{definition}

\begin{proposition}\label{prop:normeq}
Let $\O$ and $\O'$ be orbit polytopes.
Then $\O$ and $\O'$ are normally equivalent if and only if they have the same composition.
\end{proposition}
\begin{proof}
Let $v$ and $v'$ be the unique vertices of $\O$ and $\O'$, respectively, that lie in the fundamental chamber $\fundom$ of the braid fan.

$(\Longrightarrow$): Suppose $\O$ and $\O'$ have different compositions.
Then $v$ and $v'$ lie in the interior of different faces of $\fundom$.
By Lemma \ref{lem:vertexcone}, this means that $\mathcal{N}_{\O}(v)\neq\mathcal{N}_{\O'}(v')$, so $\O\not\equiv\O'$.

($\Longleftarrow$): Suppose $\O$ and $\O'$ have the same composition.
Then $v$ and $v'$ lie in the interior of the same face of $\fundom$.
So we have a bijection between the vertices of $\O=\O(v)$ and $\O'=\O(v')$ given by $\sigma(v)\mapsto\sigma(v')$ for each vertex $\sigma(v)$ of $\O$, where $\sigma\in \sym_I$.
Now let $F\le\O$, so $F=\conv\{\sigma_1(v)\dots,\sigma_m(v)\}$ for some $\sigma_1,\dots,\sigma_m\in \sym_I$.
Define $F'\subseteq\O'$ to be $\conv\{\sigma_1(v'),\dots,\sigma_m(v')\}$.
Then by the rearrangement inequality (\ref{ineq:rearrangement_inequality}), any $y\in\mathcal{N}_{\O}(F)$ attains its maximal value for $\O'$ on the vertices $\sigma_1(v'),\dots,\sigma_m(v')$ and no others, and hence on all of $F'$ and nowhere else on $\O'$.
Thus $F'$ is a face of $\O'$, and we have shown that $\mathcal{N}_{\O}(F)\subseteq\mathcal{N}_{\O'}(F')$.
The other inclusion of normal cones also follows from the rearrangement inequality.

So far, we have constructed an injection from the faces of $\O$ to the faces of $\O'$, and we have shown that this injection preserves normal cones.
Reversing the roles of $\O$ and $\O'$ shows that this map is actually a bijection.
Thus $\mathcal{N}_{\O}=\mathcal{N}_{\O'}$.
\end{proof}

\begin{corollary}
Normal equivalence classes of orbit polytopes $\O(p)$ for $p\in\reals I$ are in bijection with compositions of the integer $n$.
\end{corollary}

\begin{notation}
Let $\alpha$ be a composition of $n$.
We write $\O_{\alpha,I}$ for the normal equivalence class of orbit polytopes in $\reals I$ with composition $\alpha$.
If the set $I$ is clear from context, we may simply write $\O_\alpha$.
\end{notation}

\begin{example}[Normal equivalence classes for $n=3$]
Let $I = \{1,2,3\}$.
There are four compositions of the integer $3=|I|$, so there are four normal equivalence classes of orbit polytopes in $\reals I$, shown in Figure \ref{fig:OP3}.

    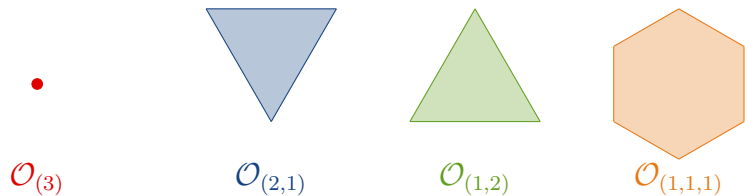
\begin{figure}
    \begin{center}
    \input{img/orbit_polytopes_3.tikz}
    \caption{Normal equivalence classes of orbit polytopes in $\reals I$ when $|I|=3$.}
    \label{fig:OP3}
    \end{center}
	\end{figure}
\end{example}

\begin{example}[Notable families of orbit polytopes]
The following compositions of $n$ correspond to normal equivalence classes of well-known families of polytopes in $\reals I$ with $|I|=n$:
    \begin{itemize}
        \item $(n)$: single point
        \item $(1,\dots,1)$: standard $n$-permutahedron
        \item $(1,n-1)$: standard $n$-simplex
        \item $(k,n-k)$: uniform matroid polytope $U_{k,n}$; these are also known as hypersimplices
    \end{itemize}
\end{example}

\section{Algebraic Structures on Orbit Polytopes}\label{sec:algebraic}

In \cite{AA}, the authors introduce a product and coproduct which give generalized permutahedra the structure of a Hopf monoid.
When restricted to orbit polytopes, these operations have a neat interpretation in terms of compositions.

\subsection{Toward a Product}\label{sec:product}
	
\begin{definition}\label{def:product}
Let $P\subset\reals S$ and $Q\subset\reals T$ be any polytopes.
Then the \textit{product of $P$ and $Q$} is the polytope
	\[P\cdot Q:=\{(p,q)\in\reals (S\sqcup T):p\in P, q\in Q\}\subset\reals (S\sqcup T).\]
The identity of this product is the unique empty orbit polytope, which lives in $\reals \varnothing$.
\end{definition}

\begin{proposition}\label{prop:uniquedecomposition}
Let $P\subset\reals I$ be a product of finitely many orbit polytopes.
Then up to commutativity of $\cdot$, $P$ has a unique expression of the form
	\[P = \O(p_1)\cdot\ldots\cdot\O(p_k)\]
where $p_j\in\reals S_j$ for $1\le j\le k$, and $S_1\sqcup\dots\sqcup S_k$ is a partition of $I$ into nonempty sets, and $\O(p_j)$ is not a single point unless $|S_j|=1$.
\end{proposition}
\begin{proof}
The vertices of the product of two polytopes are exactly the products of vertices of the two polytopes.
This means that the vertex set of a product of nonempty orbit polytopes will \textit{not} contain the entire orbit of any point under the action of the symmetric group, unless we are multiplying points with all coordinates equal.
Thus orbit polytopes that are not points cannot decompose as products of orbit polytopes.

It is a general fact about polytopes (indeed, about subsets of $\reals I$) that if $P\subset \reals I$ can be written as a product of polytopes with respect to the decomposition $I=S_1\sqcup\dots\sqcup S_k$ and also with respect to the decomposition $I=T_1\sqcup\dots\sqcup T_\ell$, then $P$ can be written as a product with respect to the common refinement of these two decompositions.
Hence if $P$ is a product of finitely many orbit polytopes with respect to the decomposition $I=S_1\sqcup\dots \sqcup S_k$, then it cannot also be a product of orbit polytopes with respect to a different decomposition, since orbit polytopes are indecomposable with respect to this product.
The exception is orbit polytopes that are points, but requiring that all multiplicands that are points live in one dimension guarantees uniqueness.
\end{proof}

\subsection{Toward a Coproduct}\label{sec:coproduct}

Each face of an orbit polytope decomposes as a product of lower-dimensional orbit polytopes.
In particular, each facet of an orbit polytope decomposes as a product of two orbit polytopes.
This will allow us to define a coproduct of orbit polytopes in Section \ref{sec:hopfmonoid}.
The following proposition makes this observation rigorous.

\begin{propdef}\label{prop:coproduct}
Let $p\in\reals I$ where $|I|=n$.
Let $\O=\O(p)\subset\reals I$ be the orbit polytope of $p$ and let $I=S\sqcup T$.
Suppose that $F\le\O$ is the face of $\O$ maximizing the indicator functional $\1_S$ of $S$, where $\1_S(x):=\sum_{s\in S}x_s$.
Then there exist unique orbit polytopes $\O|_S\subset \reals S$ and $\O/_S\subset \reals T$ such that
	\[F = \O|_S\cdot \O/_S.\]
We call $\O|_S$ ``$\O$ restricted to $S$,'' and we call $\O/_S$ ``$\O$ contracted by $S$.''
\end{propdef}
\begin{proof}
Suppose $I=\{i_1,\dots,i_n\}$.
Let $p=(p_i)_{i\in I}$ and assume without loss of generality that $p_{i_1}\ge\dots\ge p_{i_n}$.
The $1_S$-maximal face of $\O$ will be of the form $F=\conv\{v_1,\dots,v_\ell\}$, where $v_1,\dots,v_\ell$ are the elements of the $\sym_I$-orbit of $p$ that have the largest $|S|$ coordinates in the positions in $S$.
This means that the $v_j$'s are exactly the elements of the orbit of $p$ that are contained in the product of orbits
    \[\{\sigma(p_{i_1},\dots,p_{i_{|S|}}):\sigma\in \sym_S\}\cdot\{\sigma(p_{i_{|S|+1}},\dots,p_{i_n}):\sigma\in \sym_T\}.\]
Define $\O|_S$ to be $\conv\{ \sigma(p_{i_1},\dots,p_{i_{|S|}}):\sigma\in \sym_S \}\subset\reals S$, and define $\O/_S$ to be $\conv\{ \sigma(p_{i_{|S|+1}},\dots,p_{i_n}):\sigma\in \sym_T \}\subset \reals T$.
Then $F = \O|_S\cdot \O/_S$.
	
The uniqueness of this expression follows from Proposition \ref{prop:uniquedecomposition}.
\end{proof}

It is straightforward to show that if $\O$ is a product of finitely many orbit polytopes, a modified version of Proposition \ref{prop:coproduct} still holds.
That is, the $\1_S$-maximal face of $\O$ decomposes as $\O|_S\cdot\O/_S$, where $\O|_S$ is a finite product of orbit polytopes that lives in $\reals S$ and $\O/_S$ is a finite product of orbit polytopes that lives in $\reals T$.

\subsection{Species}

Let $\mathsf{Set}$ be the category of sets with arbitrary morphisms, and let $\mathsf{Set^\times}$ be the category of finite sets with bijections.

\begin{definition}
A \textit{set species} is a functor $\spF:\mathsf{Set^\times}\to\mathsf{Set}$.
If $I$ is a finite set, then $\spF$ maps $I$ to a set $\spF[I]$ which can be considered to contain ``structures of type $\spF$ labeled by $I$.''
If $\sigma:I\to J$ is a bijection of finite sets, then $\spF$ maps $\sigma$ to a morphism $\spF[\sigma]:\spF[I]\to\spF[J]$ which can be thought of as the map ``relabeling the elements of $\spF[I]$ according to $\sigma$.''
\end{definition}

\begin{definition}
The \textit{set species of orbit polytopes}, denoted $\spOP$, maps a finite set $I$ to the set $\spOP[I]$ of finite products of orbit polytopes living in $\reals I$.
For a bijection of finite sets $\sigma:I\to J$, we get the map $\spOP[\sigma]:\spOP[I]\to \spOP[J]$ induced from the isomorphism from $\reals I$ to $\reals J$ relabelling the basis vectors $\{e_i:i\in I\}$ of $\reals I$ according to $\sigma$.
\end{definition}

We have seen that orbit polytopes up to normal equivalence are in bijection with compositions.
It is interesting to consider the species of orbit polytopes up to normal equivalence.

\begin{definition}\label{def:opsetspecies}
The \textit{set species of normal equivalence classes of orbit polytopes}, denoted $\OPbar$, maps a finite set $I$ to the set $\OPbar[I]$ of normal equivalence classes of finite products of orbit polytopes in $\spOP[I]$.
In other words, a general element of $\OPbar[I]$ has the form 
	\[\O_{\alpha_1,S_1}\cdot \dots \cdot\O_{\alpha_k,S_k}\]
where $I=S_1\sqcup\dots\sqcup S_k$ and $\alpha_i$ is a composition of $|S_i|$ for all $i$.
\end{definition}

Let $\Comp$ be the set species of compositions where $\Comp[I]$ is the set of integer compositions of $|I|$ and $\Comp[\sigma] = \id$ for all bijections $\sigma:I\to J$.
Let $\Comphat$ be the result of removing compositions with one part from $\Comp[I]$ when $|I|\ge2$, so $\Comphat[I]$ for $|I|\ge 2$ is the set of compositions of $|I|$ with more than one part.
Define $\Comphat[\varnothing]:=\varnothing$.
To describe the relationship of $\OPbar$ to $\Comphat$, we need the following two definitions.

\begin{definition}
The \textit{exponential species} $\mathrm{E}$ maps a finite set $I$ to a set $\{1\}$ containing one element.
For each bijection $\sigma:I\to J$, we have $\mathrm{E}[\sigma] = \id$.
\end{definition}

\begin{definition}\label{def:speciescomp}
Let $\spF$ and $\spG$ be two set species where $\spG[\varnothing]=\varnothing$.
The \textit{composition} of $\spF$ and $\spG$ is the set species $\mathrm{F}\circ\mathrm{G}$ where
    \[\mathrm{F}\circ\mathrm{G}[I] = \bigsqcup_{X\vdash I}\bigg(F[X]\times\prod_{S\in X}G[S]\bigg)\]
where $X\vdash I$ denotes that $X$ is an unordered partition of the set $I$ into nonempty parts.
In other words, the set $\spF\circ \mathrm{G}[I]$ is the set of structures obtained by partitioning I, putting a structure of type $\mathrm{G}$ on each part, and then putting a structure of type $\spF$ on the set of parts of the partition.
\end{definition}

It is often useful to consider what happens when the exponential species is composed with some other species $\spF$.
By Definition \ref{def:speciescomp}, the elements of $\spE\circ\spF[I]$ will be the result of partitioning the set $I$ into unordered parts, then putting a structure of type $\spF$ on each part.

\begin{proposition} \label{prop:species OPbar}
The species $\OPbar$ is isomorphic to $\mathrm{E}\circ\Comphat$, where $\mathrm{E}$ is the exponential species.
\end{proposition}
\begin{proof}
This is immediate from Definition \ref{def:opsetspecies} and the commutativity of the polytope product.
\end{proof}

The exponential generating function of a set species $\spF$ is
    \[\spF(t) = \sum_{k=0}^\infty \frac{1}{k!}\bigg|\spF\Big[[k]\Big]\bigg|t^k,\]
the power series for which the coefficient of $t^k/k!$ is the number of structures of type $\spF$ on a $k$-element set.

\begin{example}
Since the exponential species has one structure on each set, its generating function is
    \[\mathrm{E}(t) = \sum_{k=0}^\infty \frac{t^k}{k!} = e^t.\]
\end{example}

The exponential generating function for the species $\spF\circ\spG$ can be obtained by composing the generating functions of $\spF$ and $\spG$ \cite[\S 2.2]{BLL}.

\begin{corollary}
The exponential generating function of $\OPbar$ is
	\[\OPbar(t) = e^{\frac{1}{2}e^{2t}-e^t+t+\frac{1}{2}}.\]
\end{corollary}
\begin{proof}
The exponential generating function for $\Comphat$ is
	\begin{align*}
	\Comphat(t) &= 0+t+\sum_{k=2}^\infty \frac{2^{k-1}-1}{k!}t^k\\\
	&= \frac{1}{2}e^{2t}-e^t+t+\frac{1}{2}
	\end{align*}
Since $\OPbar = \spE\circ\Comphat$, we get the result by composing the function $e^t$ with this generating function for $\Comphat$.
\end{proof}
	
The first few terms of the generating function for $\OPbar$ are
    \[\OPbar(t) = 1 + t + 2\frac{t^2}{2!} + 7\frac{t^3}{3!}+ 29\frac{t^4}{4!} + 136\frac{t^5}{5!} + \dots\]

\subsection{The Hopf Monoid of Orbit Polytopes}\label{sec:hopfmonoid}

We now define the main algebraic object of interest.

\begin{definition}\label{def:hopfmonoid}
A \textit{Hopf monoid in set species} is a set species $\rH$ equipped with a collection of product maps $\mu=\{\mu_{S,T}:\rH[S]\times\rH[T]\to\rH[I]\}$ and a collection of coproduct maps $\Delta = \{\Delta_{S,T}:\rH[I]\to\rH[S]\times\rH[T]\}$ where $S$ and $T$ are any pair of disjoint finite sets and $I=S\sqcup T$.
These operations must satisfy naturality, unitality, associativity, and compatibility axioms (see \cite{AM,AA}).
A Hopf monoid in set species is \textit{connected} if $|\rH[\varnothing]|=1$.
\end{definition}

\begin{proposition}[$\spOP$ is a Hopf submonoid of $\rGP$]\label{prop:opHopf}
Define a product and coproduct on $\spOP$ as follows:
	\begin{itemize}
	\item The product is a collection of maps $\mu=\{\mu_{S,T}:\spOP[S]\times \spOP[T]\to \spOP[I]\}$ for all ordered partitions of a finite set $I$ into finite sets $S$ and $T$.
	If $\O\in\spOP[S]$ and $\O'\in\spOP[T]$, then their product is
		\[\mu_{S,T}(\O,\O') := \O\cdot\O'\in\spOP[I]\]
	as defined in Definition \ref{def:product}.
	\item The coproduct is a collection of maps $\Delta = \{\Delta_{S,T}:\spOP[I]\to\spOP[S]\times\spOP[T]\}$ for all ordered partitions of a finite set $I$ into finite sets $S$ and $T$.
	If $\O\in\spOP[I]$, then its coproduct is
		\[\Delta_{S,T}(\O) := (\O|_S,\O/_S)\in\spOP[S]\times\spOP[T],\]
	where $\O|_S$ and $\O/_S$ are the restriction and contraction discussed in Proposition \ref{prop:coproduct}.
	\end{itemize}
These operations turn the set species $\spOP$ into a connected Hopf submonoid of $\rGP$, where $\rGP$ is the Hopf monoid of generalized permutahedra defined in \cite{AA}.
\end{proposition}
\begin{proof}
The product and coproduct described here are the same as the product and coproduct on $\mathrm{GP}$ described in \cite{AA}.
Products and coproducts of elements in $\spOP$ result in elements of $\spOP$ (see Sections \ref{sec:product} and \ref{sec:coproduct}), so $\spOP$ is closed under this product and coproduct.
Thus $\spOP$ is a Hopf submonoid of $\mathrm{GP}$.
\end{proof}

\begin{proposition}[\cite{AA}]\label{prop:normalrespects}
Taking normal equivalence classes respects the product and coproduct of $\spOP$ defined in Proposition \ref{prop:opHopf}.
\end{proposition}
	 
\begin{corollary}
The set species $\OPbar$ of normal equivalence classes of orbit polytopes forms a connected Hopf monoid under the induced product and coproduct from $\spOP$.
\end{corollary}

As a consequence of Proposition \ref{prop:uniquedecomposition}, we get that $\OPbar$ is a free commutative Hopf monoid generated under multiplication by elements $\O_{\alpha,I}$ where $I$ is some finite set and $\alpha\in\Comphat[I]$.
This characterizes the product of $\OPbar$.
The coproduct of $\OPbar$ also has a very nice formulation in terms of compositions.
This formulation uses two standard operations.

\begin{definition}
The \textit{concatenation} of the compositions $\beta=(\beta_1,\dots,\beta_k)$ and $\gamma = (\gamma_1,\dots,\gamma_\ell)$ is the composition
    \[\beta\concat\gamma := (\beta_1,\dots,\beta_k,\gamma_1,\dots,\gamma_\ell).\]
\end{definition}

\begin{definition}
The \textit{near-concatenation} of nonempty compositions $\beta=(\beta_1,\dots,\beta_k)$ and $\gamma = (\gamma_1,\dots,\gamma_\ell)$ is the composition
    \[\beta\nconcat\gamma := (\beta_1,\dots,\beta_{k-1},\beta_k+\gamma_1,\gamma_2,\dots,\gamma_\ell).\]
\end{definition}

\begin{proposition}\label{prop:composition coproduct}
Let $I$ be a finite set with $|I|=n$ and let $\alpha$ be an integer composition of $n$, so $\O_{\alpha}\in\OPbar[I]$.
Then if $I = S\sqcup T$ we have
    \[\Delta_{S,T}(\O_{\alpha}) = (\O_{\alpha|_S},\O_{\alpha/_S})\]
where $\alpha|_S$ and $\alpha/_S$ are the unique pair of compositions satisfying
    \begin{enumerate}[(i)]
        \item\label{item:1} $\alpha|_S$ is a composition of $|S|$ and $\alpha/_S$ is a composition of $|T|$, and
        \item\label{item:2} either $\alpha|_S\concat\alpha/_S=\alpha$ or $\alpha|_S\nconcat\alpha/_S=\alpha$.
    \end{enumerate}
\end{proposition}
\begin{proof}
Suppose that $\O(p)$ is in the normal equivalence class $\O_\alpha$ where $\alpha=(\alpha_1,\dots,\alpha_k)$.
This means that the point $p$ has $\alpha_1$ occurrences of the largest coordinate, $\alpha_2$ occurrences of the second largest coordinate, and so on.
Suppose $p$ has coordinates $p_1,\dots,p_n$ with $p_1\ge\dots\ge p_n$, and define $q:=(p_1,\dots,p_{|S|})\in\reals S$ and $q':=(p_{|S|+1},\dots,p_{|T|})\in\reals T$.
Since $q$ and $q'$ are obtained from the $|S|$ largest and $|T|$ smallest coordinates of $p$, respectively, we can see that the compositions of these two points satisfy conditions (\ref{item:1}) and (\ref{item:2}).

Now consider $\Delta_{S,T}(\O(p))$.
From the proof of Proposition \ref{prop:coproduct} and the definition of $\Delta_{S,T}$ in Proposition \ref{prop:opHopf}, we know that \[\Delta_{S,T}(\O(p))=(\O(p)|_S,\O(p)/_S)\] where $\O(p)|_S=\O(q)\subset\reals S$ and $\O(p)/_S=\O(q')\subset\reals T$.
Passing to normal equivalence classes and invoking Proposition \ref{prop:normalrespects} gives us the result.
\end{proof}

Hopf monoids are required to be associative, so taking products of more than two elements is well-defined.
Likewise, coassociativity implies that decomposing an element along more than two sets is well-defined.
The following proposition follows immediately.

\begin{proposition}\label{prop:higher products and coproducts}
Let $I$ be a finite set and let $S_1\sqcup\dots\sqcup S_k$ be an ordered partition of $I$.
    \begin{itemize}
        \item The map $\mu_{S_1,\dots,S_k}:\OPbar[S_1]\times\dots\times\OPbar[S_k]\to\OPbar[I]$, obtained from iterating the product of $\OPbar$ and invoking associativity, is given by 
            \[\mu_{S_1,\dots,S_k}(\O_{\alpha_1,S_1},\dots,\O_{\alpha_k,S_k}) = \O_{\alpha_1,S_1}\cdot\dots\cdot\O_{\alpha_k,S_k}.\]
        \item The map $\Delta_{S_1,\dots,S_k}:\OPbar[I]\to\OPbar[S_1]\times\dots\times\OPbar[S_k]$, obtained from iterating the coproduct of $\OPbar$ and invoking coassociativity, is given by
            \[\Delta_{S_1,\dots,S_k}(\O_{\alpha,I}) = (\O_{\alpha|_{S_1},S_1},\dots,\O_{\alpha|_{S_k},S_k})\]
        where $\alpha|_{S_1},\dots,\alpha|_{S_k}$ is the unique sequence of compositions such that $\alpha|_{S_i}$ is a composition of $|S_i|$ for all $i$ and $\alpha$ can be obtained from the $\alpha|_{S_i}$'s by some sequence of concatenations and near-concatenations; that is,
            \[\alpha = \alpha|_{S_1}\Box \dots\Box\alpha|_{S_k}\]
        where each occurrence of $\Box$ is replaced with either concatenation $\cdot$ or near-concat\-enation $\odot$.
    \end{itemize}
\end{proposition}

\subsection{The Hopf Algebra of Compositions}
\label{sec:hopfalgebracompositions}

Given a Hopf monoid in set species, we can obtain a Hopf algebra by applying first a linearization functor and then a Fock functor \cite[\S 15]{AM}.
The first Fock functor produces the graded Hopf algebra
    \[\bigoplus_{n\ge0} \textrm{Span}\{\text{isomorphism classes of elements of }\rH[I] \text{ where }|I|=n\}\]
where isomorphisms are given by relabeling maps in the species.
For orbit polytopes, these isomorphism classes can be described using compositions.
Let $A$ be the set containing all integer compositions with more than one part and the unique composition of $1$.
We will use the notation $|\alpha|$ to denote the sum of the parts of a composition $\alpha$.

\begin{lemma}\label{lem:isomorphism classes}
Isomorphism classes of elements of $\OPbar[I]$ where $|I|=n$ are in bijection with multisets of compositions $\alpha_1,\dots,\alpha_k$ where $\sum_{i=1}^k|\alpha_i|=n$ and $\alpha_i\in A$ for all $i$.
\end{lemma}
\begin{proof}
Let $I$ and $J$ be finite sets with $|I|=|J|$.
We know that orbit polytopes are invariant under the action of the symmetric group, so their normal fans must be invariant as well.
Thus the normal equivalence classes $\O_{\alpha,I}$ and $\O_{\beta,J}$ will be isomorphic if and only if they correspond to the same coarsening of the braid fan, that is, if $\alpha=\beta$.

Now consider general elements $\O\in\OPbar[I]$ and $\O'\in\OPbar[J]$.
Proposition \ref{prop:uniquedecomposition} implies that up to commutativity, we can uniquely write products $\O=\O_{\alpha_1,S_1}\cdot\dots\cdot\O_{\alpha_k,S_k}$ and $\O'=\O_{\beta_1,T_1}\cdot\dots\cdot\O_{\beta_\ell, T_\ell}$ where $\alpha_i,\beta_j\in A$ for all $i$ and $j$.
Combined with uniqueness, our observation from the previous paragraph implies that $\O$ and $\O'$ are isomorphic if and only if $k=\ell$ and $\{\alpha_1,\dots,\alpha_k\}=\{\beta_1,\dots,\beta_\ell\}$ as multisubsets of $A$.
Noting that $\sum_{i=1}^k|\alpha_i|=n=\sum_{j=1}^\ell|\beta_j|$ completes the proof.
\end{proof}

Applying the first Fock functor to $\OPbar$ results in a Hopf algebra of normal equivalence classes of orbit polytopes.
We can use Lemma \ref{lem:isomorphism classes} to characterize this Hopf algebra using compositions.
	
\begin{theorem}[Hopf algebra  of compositions]\label{thm:hopfalgebra}
Consider the commutative algebra $\Hcomp$ generated freely by $A$.
Let $\alpha\in A$ and define a coproduct $\Delta:\Hcomp\to\Hcomp\otimes\Hcomp$ by
	\[\Delta(\alpha) := \sum_{\substack{\beta\concat\gamma=\alpha\\ \text{or}\\\beta\nconcat \gamma=\alpha}}\binom{|\alpha|}{|\beta|}\beta\otimes\gamma,\]
where the composition $(n)$ is defined to be equal to the product of compositions $(1)^n$.
This makes $\Hcomp$ into a graded Hopf algebra isomorphic to the Hopf algebra of normal equivalence classes of orbit polytopes.
\end{theorem}
\begin{proof}
Let $[\alpha]$ denote the isomorphism class of $\O_\alpha$.
Lemma \ref{lem:isomorphism classes} shows that the Hopf algebra of orbit polytopes is a commutative algebra generated under multiplication by elements $[\alpha]$ for $\alpha\in A$.
There are no further relations among these elements.
Next, let us consider the coproduct in the Hopf algebra of orbit polytopes.
We can apply the formula from \cite[Section \S 2.9]{AA} to see that for a composition $\alpha$,
    \[\Delta([\alpha]) = \sum_{[|\alpha|]=S\sqcup T}[\alpha|_S]\otimes[\alpha/_S].\]
But $\alpha|_S$ and $\alpha/_S$ depend only on $|S|$, so we have
    \[\Delta([\alpha]) = \sum_{i=0}^{|\alpha|}\binom{|\alpha|}{i}[\alpha|_{\{1,\dots,i\}}]\otimes[\alpha/_{\{1,\dots,i\}}].\]
The pairs of compositions arising as $a|_{\{1,\dots,i\}}$ and $\alpha/_{\{1,\dots,i\}}$ are exactly all of the pairs of compositions $\beta$ and $\gamma$ such that $i=|\beta|$ and either $\beta\concat\gamma=\alpha$ or $\beta\nconcat\gamma=\alpha$.
Thus we have
    \[\Delta([\alpha]) := \sum_{\substack{\beta\concat\gamma=\alpha\\ \text{or}\\\beta\nconcat \gamma=\alpha}}\binom{|\alpha|}{|\beta|}[\beta]\otimes[\gamma].\]
Furthermore, we note that the isomorphism class $[(n)]$ should correspond under Lemma \ref{lem:isomorphism classes} to the multiset containing $n$ copies of the composition $(1)$.
Hence the map from the Hopf algebra of orbit polytopes to $\Hcomp$ sending $[\alpha]$ to $\alpha$ is an isomorphism of graded Hopf algebras.
\end{proof}

\begin{example}
We compute the coproduct of the composition $(1,2,1)\in\Hcomp$:
    \begin{align*}
    \Delta([(1,2,1)]) = &[\varnothing]\otimes[(1,2,1)]\\
    + 4\cdot&[(1)]\otimes[(2,1)]\\
    + 6\cdot&[(1,1)]\otimes[(1,1)]\\ 
    + 4\cdot&[(1,2)]\otimes[(1)]\\ 
    + &[(1,2,1)]\otimes[\varnothing].
    \end{align*}
\end{example}

\section{Characters on \texorpdfstring{$\OPbar$}{OP}}\label{sec:character}

\subsection{Character Group}

Studying the character group of a Hopf monoid can lead to surprising connections, as seen for the cases of permutahedra and associahedra in \cite{AA}.
To begin, let $\rH$ be a connected Hopf monoid in set species, and let $\Bbbk$ be a field.

\begin{definition}
A \textit{character} $\zeta:\rH\to\Bbbk$ is collection of natural maps $\{\zeta_I:\rH[I]\to\Bbbk\}$ for each finite set $I$ such that
    \begin{enumerate}[(i)]
        \item $\zeta_\varnothing:\rH[\varnothing]\to\Bbbk$ is the map sending $1\in\rH[\varnothing]$ to $1\in\Bbbk$, and
        \item if $I=S\sqcup T$, then for any $x\in\rH[S]$ and $y\in\rH[T]$ we have
            \[\zeta_I(\mu_{S,T}(x,y)) = \zeta_S(x)\zeta_T(y).\]
    \end{enumerate}
\end{definition}

Let $\mathbb{X}(\rH)$ be the set of all characters on $\rH$.

\begin{definition}
The \textit{convolution} of $\zeta,\psi\in\mathbb{X}(\rH)$ is defined for $x\in\rH[I]$ to be
    \[(\zeta\star\psi)_I(x) := \sum_{I=S\sqcup T} \zeta_S(x|_S)\psi_T(x/_S) \]
where the sum is taken over all ordered partitions of $I$ into sets $S$ and $T$.
\end{definition}

This convolution product gives $\mathbb{X}(\rH)$ a group structure \cite[Theorem 8.2]{AA}.
We will prove that the character group of $\OPbar$ is related to the Hopf algebra  $NSym$ of noncommutative symmetric functions.
This is the graded dual of the Hopf algebra $QSym$ of quasisymmetric functions.
One basis for $NSym$ is given by the \textit{noncommutative ribbon functions} $\{R_\alpha\}$ which are indexed by integer compositions $\alpha$.
This basis is dual to the fundamental basis of $QSym$ and has the product
    \begin{equation}\label{eq:product in NSym}
        R_\beta R_\gamma = R_{\beta\concat\gamma}+R_{\beta\nconcat\gamma}.
    \end{equation}
A detailed explanation of the Hopf structures of $QSym$ and $NSym$ can be found in \cite{RG}.

\begin{definition}
The \textit{completion} $\overline{NSym}$ of $NSym$ is the ring $\Bbbk[[\{R_\alpha\}]]$ of generating functions of the form
    \[\sum_\alpha c_\alpha R_\alpha\]
where the sum is over compositions $\alpha$.
The product in this ring is induced from the product of the $R_\alpha$ given in Equation \ref{eq:product in NSym}.
\end{definition}

One can show that an element of $\overline{NSym}$ is invertible if and only if $c_\varnothing\ne0$, and that the invertible elements form a group under multiplication.

\begin{definition}
Define $G$ to be the collection of invertible elements in $\overline{NSym}$ with the properties that $c_\varnothing=1$ and $n!c_{(n)}=c_{(1)}^n$ for all $n>1$.
\end{definition}

It is straightforward to check that $G$ is a subgroup of the invertible elements of $\overline{NSym}$.

\begin{theorem}\label{prop:mainresult}
The character group $\mathbb{X}(\OPbar)$ is isomorphic to $G$.
\end{theorem}
\begin{proof}
Since characters are natural maps, their values are unchanged under the relabeling action of the species $\OPbar$.
Thus for $\zeta\in\mathbb{X}(\OPbar)$, it suffices to define only $\zeta_{[n]}$ for each positive integer $n$, since each other finite set can be relabeled to one of these.
(The definition of a character takes care of $\zeta_\varnothing$, requiring $\zeta_\varnothing(\O_\varnothing)=1$.)

Define $\comps$ to be the set of all integer compositions, including the empty composition $\varnothing$.
(Note: This set should not be confused with $\Hcomp$, the Hopf algebra of compositions defined in Section \ref{sec:hopfalgebracompositions}.)
For each positive integer $n$, define $\compn$ to be the set of all compositions of $n$.
Since characters are multiplicative functions, it suffices to define the values of the character only on a generating set.
So for $\zeta\in\mathbb{X}(\OPbar)$, we only need to determine $\zeta_{[n]}(\O_\alpha)$ for each $n\ge1$ and for each $\alpha\in\compn$.
A point $\O_{(n)}$ in $\reals^n$ is a point in $\reals$ raised to the $n$th power, so we must have $\zeta_{[n]}(\O_{(n)})=(\zeta_{[1]}(\O_{(1)}))^n$.
When $\alpha$ is the composition $(1)\in\comps_1$ or a composition with more than one part, we may freely choose the value of $\zeta$ on $\O_\alpha$, since the equivalence classes corresponding to these compositions form a free commutative generating set for $\OPbar$.

Let $\zeta\in\mathbb{X}(\OPbar)$ and let $b_\alpha:=\zeta_{[n]}(\O_\alpha)$ for each $\alpha\in\comps$, where $n=|\alpha|$.
The sequence $(b_\alpha:\alpha\in\comps)$ contains all of the information needed to define $\zeta$.
Let
    \[F:\mathbb{X}(\OPbar)\to\overline{NSym}\]
be the map sending $\zeta$ to the exponential generating function
    \[F(\zeta) := \sum_{\beta\in\comps}\frac{b_\beta}{|\beta|!}R_\beta\]
where $R_\beta\in NSym$ is the noncommutative ribbon function corresponding to the composition $\beta$.
Our analysis in the previous paragraph shows that $F$ actually induces a bijection between $\mathbb{X}(\OPbar)$ and the subgroup $G$ of $\overline{NSym}$.

Let $\psi\in\mathbb{X}(\OPbar)$ as well, so $\psi$ has a generating function
    \[F(\psi) = \sum_{\gamma\in\comps}\frac{c_\gamma}{|\gamma|!}R_\gamma.\]
Then $\zeta\star\psi$ also has some generating function
    \[F(\zeta\star\psi) = \sum_{\alpha\in\comps}\frac{a_\alpha}{|\alpha|!}R_\alpha.\]
The convolution product of characters tells us that for $\alpha\in\compn$, we should have
    \begin{align*}
        a_\alpha &= (\zeta\star\psi)_{[n]}(\O_\alpha)\\
        &= \sum_{[n]=S\sqcup T}\zeta_S(\O_{\alpha|_S})\psi_T(\O_{\alpha/_S})\\
        &= \sum_{\substack{\beta\concat\gamma=\alpha\\ \text{or}\\\beta\nconcat \gamma=\alpha}} 
        \binom{|\alpha|}{|\beta|}b_\beta c_\gamma.
    \end{align*}

Multiplying the generating functions for $\zeta$ and $\psi$ and using the product of the ribbon functions, we get
    \begin{align*}
        F(\zeta)F(\psi)&= \sum_{\beta,\gamma\in\comps}\frac{b_\beta c_\gamma}{|\beta|!|\gamma|!}R_\beta R_\gamma\\
        &= \sum_{\beta,\gamma\in\comps}\frac{b_\beta c_\gamma}{|\beta|!|\gamma|!}(R_{\beta\concat\gamma}+R_{\beta\nconcat\gamma})\\
        &= \sum_{\alpha\in\comps}\Big(\sum_{\substack{\beta\concat\gamma=\alpha\\ \text{or}\\\beta\nconcat \gamma=\alpha}}\frac{b_\beta c_\gamma}{|\beta|!|\gamma|!}\Big)R_\alpha\\
        &= \sum_{\alpha\in\comps}\frac{1}{|\alpha|!}\bigg(\sum_{\substack{\beta\concat\gamma=\alpha\\ \text{or}\\\beta\nconcat \gamma=\alpha}}\binom{|\alpha|}{|\beta|}b_\beta c_\gamma\bigg)R_\alpha\\
        &= F(\zeta\star\psi).
    \end{align*}
Thus the multiplication of these generating functions corresponds to the convolution of characters in $\mathbb{X}(\OPbar)$.
This means that $\mathbb{X}(\OPbar)$ and $G$ are isomorphic as groups.
\end{proof}

\subsection{Basic Character}\label{sec:basiccharacter}

A special character on generalized permutahedra is related to many well-studied invariants, such as the chromatic polynomial for graphs and the Billera-Jia-Reiner polynomial for matroids.
Here we examine what this character looks like for orbit polytopes.

\begin{definition}\cite{AA}
The \textit{basic character} $\zeta$ on the Hopf monoid of generalized permutahedra is given by 
    \[\zeta_I(P) = \begin{cases}1 &\text{if }P\text{ is a point}\\
    0&\text{otherwise}\end{cases}\]
for a generalized permutahedron $P\in\rGP[I]$.
\end{definition}
	
If we restrict the basic character to $\OPbar$, we can interpret it using compositions.

\begin{proposition}
Restricting the basic character of generalized permutahedra to $\OPbar$ produces the character $\zeta\in\mathbb{X}(\OPbar)$ defined by
    \[\zeta_I(\O_\alpha) = \begin{cases}1 &\text{if }\alpha\text{ has only one part}\\
    0 &\text{otherwise}\end{cases}\]
when $\alpha$ is an integer composition of $|I|$.
\end{proposition}
\begin{proof}
The orbit polytope $\O_\alpha$ is a point if and only if $\alpha$ is a composition with one part.
\end{proof}

\subsection{Polynomial Invariant of the Basic Character}

Given a character $\zeta$ on a Hopf monoid $\rH$ and an element $x\in \rH[I]$, we can obtain a polynomial invariant $\chi_I(x)$ given by
    \[\chi_I(x)(t) = \sum_{k=1}^{|I|}\bigg(\sum_{(S_1,\dots,S_k)\vDash I}(\zeta_{S_1}\otimes\cdots\otimes\zeta_{S_k})\circ\Delta_{S_1,\dots,S_k}(x)\bigg)\binom{t}{k}.\]
Here, we are summing over ordered partitions $(S_1,\dots,S_k)$ of the set $I$ such that all of the $S_i$ are nonempty.

A polynomial invariant of particular interest is the one corresponding to the basic character.
For graphs, this invariant is the chromatic polynomial; for matroids, it is the Billera-Jia-Reiner polynomial \cite[\S 18]{AA}.
For orbit polytopes, we can interpret this invariant in terms of compositions, but first we need some extra definitions and notation.
Given an integer $n$ and a composition $\gamma = (\gamma_1, \dots,\gamma_k)$ of $n$, let 
    \[\binom{n}{\gamma}:=\binom{n}{\gamma_1,\dots,\gamma_k}.\]
We write $\ell(\gamma)$ for the number of parts of $\gamma$.
Finally, if $\alpha$ is another composition of $n$, we say that $\gamma$ \textit{refines} $\alpha$ if $\gamma$ can be obtained from $\alpha$ by splitting each part of $\alpha$ into one or more parts.

\begin{proposition}
Let $\O_\alpha\in\OPbar[I]$.
Then the polynomial invariant of $\O_\alpha$ corresponding to the basic character is given by
    \[\chi_I(\O_\alpha)(t) = \sum_{\gamma\text{ refines }\alpha}\binom{|I|}{\gamma}\binom{t}{\ell(\gamma)}.\]
\end{proposition}
\begin{proof}
Let $\zeta$ be the basic character on $\OPbar$ defined in Section \ref{sec:basiccharacter}.
Then 
    \[(\zeta_{S_1}\otimes\cdots\otimes\zeta_{S_k})\circ\Delta_{S_1,\dots,S_k}(\O_\alpha)\]
will be nonzero if and only if $\Delta_{S_1,\dots,S_k}(\O_\alpha)$ is a tuple of points.
We know from Proposition \ref{prop:higher products and coproducts} that $\Delta_{S_1,\dots,S_k}(\O_\alpha) = (\O_{\alpha|_{S_1},S_1},\dots,\O_{\alpha|_{S_k},S_k})$.
In order for this to be a tuple of points, each $\alpha|_{S_i}$ must have only one part.
This will be the case if and only if the composition $(|S_1|,\dots,|S_k|)$ refines $\alpha$.
Hence we have
    \begin{align*}
        \chi_I(x)(t) &= \sum_{k=1}^{|I|}\bigg(\sum_{(S_1,\dots,S_k)\vDash I}(\zeta_{S_1}\otimes\cdots\otimes\zeta_{S_k})\circ\Delta_{S_1,\dots,S_k}(\O_\alpha)\bigg)\binom{t}{k}\\
        &= \sum_{k=1}^{|I|}\bigg( \sum_{\substack{(S_1,\dots,S_k)\vDash I\\(|S_1|,\dots,|S_k|)\text{ refines }\alpha}} 1 \bigg)\binom{t}{k}\\
        &= \sum_{k=1}^{|I|} \bigg(\sum_{\substack{\gamma\text{ refines }\alpha\\\ell(\gamma)=k}}\binom{|I|}{\gamma}\bigg)\binom{t}{k}\\
        &= \sum_{\gamma\text{ refines }\alpha}\binom{|I|}{\gamma}\binom{t}{\ell(\gamma)},
    \end{align*}
as desired.
\end{proof}

\section*{Acknowledgements}
The author would like to thank Federico Ardila for the valuable guidance, Mario Sanchez for helpful conversations, and the reviewers of an abridged version for useful feedback.

\bibliographystyle{amsalpha}
\bibliography{supina}

\end{document}

%% file: img/normally_equivalent_polytopes.tikz
\begin{tikzpicture}[draw=ngreen,scale=1.5]

\coordinate (1) at (0,0);
\coordinate (2) at (1,0);
\coordinate (3) at (1.5,1);
\coordinate (4) at (.5,2);
\coordinate (5) at (-.5,1);

\node[color=ngreen,anchor=north] at (1) {$p_1$};
\node[color=ngreen,anchor=north] at (2) {$p_2$};
\node[color=ngreen,anchor=north west] at (3) {$p_3$};
\node[color=ngreen,anchor=south] at (4) {$p_4$};
\node[color=ngreen,anchor=north east] at (5) {$p_5$};
\node at (.5,.9) {$P$};

\draw (1)--(2);
\draw (2)--(3);
\draw (3)--(4);
\draw (4)--(5);
\draw (5)--(1);

\fill[ngreen, fill opacity=.5] (1) to (2) to (3) to (4) to (5) to cycle;

\coordinate (6) at (4,0);
\coordinate (7) at (5.75,0);
\coordinate (8) at (6,.5);
\coordinate (9) at (4.5,2);
\coordinate (10) at (3.5,1);

\node[color=ngreen,anchor=north] at (6) {$q_1$};
\node[color=ngreen,anchor=north] at (7) {$q_2$};
\node[color=ngreen,anchor=west] at (8) {$q_3$};
\node[color=ngreen,anchor=south] at (9) {$q_4$};
\node[color=ngreen,anchor=south east] at (10) {$q_5$};
\node at (4.5,.9) {$Q$};

\draw (6)--(7);
\draw (7)--(8);
\draw (8)--(9);
\draw (9)--(10);
\draw (10)--(6);

\fill[ngreen, fill opacity=.5] (6) to (7) to (8) to (9) to (10) to cycle;

\end{tikzpicture}

%% file: img/simplex.tikz
\begin{tikzpicture}%
	[edge/.style={color=ngreen},
	facet/.style={fill=ngreen,fill opacity=0.300000}]
	vertex/.style={inner sep=1pt,circle,draw=green!25!black,fill=green!75!black,thick,anchor=base}]
%
%
\coordinate (0.00000, 1.00000) at (0.00000, 1.00000);
\coordinate (0.86603, -0.50000) at (0.86603, -0.50000);
\coordinate (-0.86603, -0.50000) at (-0.86603, -0.50000);
\draw[ngreen,fill=ngreen] (0.00000, 1.00000) circle (2pt);
\node[anchor=south] at (0.00000, 1.00000) {\color{ngreen}$(1,0,0)$};
\draw[ngreen,fill=ngreen] (0.86603, -0.50000) circle (2pt);
\node[anchor=west] at (0.86603, -0.50000) {\color{ngreen}$(0,1,0)$};
\draw[ngreen,fill=ngreen] (-0.86603, -0.50000) circle (2pt);
\node[anchor=east] at (-0.86603, -0.50000) {\color{ngreen}$(0,0,1)$};
\fill[facet] (-0.86603, -0.50000) -- (0.00000, 1.00000) -- (0.86603, -0.50000) -- cycle {};
\draw[edge] (0.00000, 1.00000) -- (0.86603, -0.50000);
\draw[edge] (0.00000, 1.00000) -- (-0.86603, -0.50000);
\draw[edge] (0.86603, -0.50000) -- (-0.86603, -0.50000);
\end{tikzpicture}
\qquad
\begin{tikzpicture}%
	[edge/.style={color=dblue},
	facet/.style={fill=dblue,fill opacity=0.300000}]
	vertex/.style={inner sep=1pt,circle,draw=dblue!25!black,fill=dblue!75!black,thick,anchor=base}]
\coordinate (0.00000, -.5) at (0.00000, -.5);
\coordinate (0.86603, 1) at (0.86603, 1);
\coordinate (-0.86603, 1) at (-0.86603, 1);
\draw[ngreen,fill=dblue] (0.00000, -.5) circle (2pt);
\node[anchor=north] at (0.00000, -.5) {\color{dblue}$(0,1,1)$};
\draw[ngreen,fill=dblue] (0.86603, 1) circle (2pt);
\node[anchor=west] at (0.86603, 1) {\color{dblue}$(1,1,0)$};
\draw[ngreen,fill=dblue] (-0.86603, 1) circle (2pt);
\node[anchor=east] at (-0.86603, 1) {\color{dblue}$(1,0,1)$};
\fill[facet] (-0.86603, 1) -- (0.00000, -.5) -- (0.86603, 1) -- cycle {};
\draw[edge] (0.00000, -.5) -- (0.86603, 1);
\draw[edge] (0.00000, -.5) -- (-0.86603, 1);
\draw[edge] (0.86603, 1) -- (-0.86603, 1);
\end{tikzpicture}

%% file: img/permutahedron.tikz
\begin{tikzpicture}%
	[x={(0.767968cm, 0.559570cm)},
	y={(-0.407418cm, 0.802202cm)},
	z={(0.494203cm, -0.208215cm)},
	back/.style={loosely dotted, thick},
	edge/.style={color=norange, ultra thick},
	facet/.style={fill=norange,fill opacity=0.300000}]

\coordinate (-0.70711, -1.22474, -1.73205) at (-0.70711, -1.22474, -1.73205);
\coordinate (-0.70711, -2.04124, -0.57735) at (-0.70711, -2.04124, -0.57735);
\coordinate (-1.41421, 0.00000, -1.73205) at (-1.41421, 0.00000, -1.73205);
\coordinate (-1.41421, -1.63299, 0.57735) at (-1.41421, -1.63299, 0.57735);
\coordinate (-2.12132, 0.40825, -0.57735) at (-2.12132, 0.40825, -0.57735);
\coordinate (-2.12132, -0.40825, 0.57735) at (-2.12132, -0.40825, 0.57735);
\coordinate (0.70711, -1.22474, -1.73205) at (0.70711, -1.22474, -1.73205);
\coordinate (0.70711, -2.04124, -0.57735) at (0.70711, -2.04124, -0.57735);
\coordinate (-0.70711, 1.22474, -1.73205) at (-0.70711, 1.22474, -1.73205);
\coordinate (-0.70711, -1.22474, 1.73205) at (-0.70711, -1.22474, 1.73205);
\coordinate (-1.41421, 1.63299, -0.57735) at (-1.41421, 1.63299, -0.57735);
\coordinate (-1.41421, 0.00000, 1.73205) at (-1.41421, 0.00000, 1.73205);
\coordinate (1.41421, 0.00000, -1.73205) at (1.41421, 0.00000, -1.73205);
\coordinate (1.41421, -1.63299, 0.57735) at (1.41421, -1.63299, 0.57735);
\coordinate (0.70711, 1.22474, -1.73205) at (0.70711, 1.22474, -1.73205);
\coordinate (0.70711, -1.22474, 1.73205) at (0.70711, -1.22474, 1.73205);
\coordinate (-0.70711, 2.04124, 0.57735) at (-0.70711, 2.04124, 0.57735);
\coordinate (-0.70711, 1.22474, 1.73205) at (-0.70711, 1.22474, 1.73205);
\coordinate (2.12132, 0.40825, -0.57735) at (2.12132, 0.40825, -0.57735);
\coordinate (2.12132, -0.40825, 0.57735) at (2.12132, -0.40825, 0.57735);
\coordinate (1.41421, 1.63299, -0.57735) at (1.41421, 1.63299, -0.57735);
\coordinate (1.41421, 0.00000, 1.73205) at (1.41421, 0.00000, 1.73205);
\coordinate (0.70711, 2.04124, 0.57735) at (0.70711, 2.04124, 0.57735);
\coordinate (0.70711, 1.22474, 1.73205) at (0.70711, 1.22474, 1.73205);

\draw[edge,back] (-0.70711, -1.22474, -1.73205) -- (-0.70711, -2.04124, -0.57735);
\draw[edge,back] (-0.70711, -1.22474, -1.73205) -- (-1.41421, 0.00000, -1.73205);
\draw[edge,back] (-0.70711, -1.22474, -1.73205) -- (0.70711, -1.22474, -1.73205);
\draw[edge,back] (-0.70711, -2.04124, -0.57735) -- (-1.41421, -1.63299, 0.57735);
\draw[edge,back] (-0.70711, -2.04124, -0.57735) -- (0.70711, -2.04124, -0.57735);
\draw[edge,back] (-1.41421, 0.00000, -1.73205) -- (-2.12132, 0.40825, -0.57735);
\draw[edge,back] (-1.41421, 0.00000, -1.73205) -- (-0.70711, 1.22474, -1.73205);
\draw[edge,back] (0.70711, -1.22474, -1.73205) -- (0.70711, -2.04124, -0.57735);
\draw[edge,back] (0.70711, -1.22474, -1.73205) -- (1.41421, 0.00000, -1.73205);
\draw[edge,back] (0.70711, -2.04124, -0.57735) -- (1.41421, -1.63299, 0.57735);
\draw[edge,back] (1.41421, 0.00000, -1.73205) -- (0.70711, 1.22474, -1.73205);
\draw[edge,back] (1.41421, 0.00000, -1.73205) -- (2.12132, 0.40825, -0.57735);

\fill[facet] (1.41421, 0.00000, 1.73205) -- (0.70711, -1.22474, 1.73205) -- (1.41421, -1.63299, 0.57735) -- (2.12132, -0.40825, 0.57735) -- cycle {};
\fill[facet] (0.70711, 1.22474, 1.73205) -- (-0.70711, 1.22474, 1.73205) -- (-1.41421, 0.00000, 1.73205) -- (-0.70711, -1.22474, 1.73205) -- (0.70711, -1.22474, 1.73205) -- (1.41421, 0.00000, 1.73205) -- cycle {};
\fill[facet] (-1.41421, 0.00000, 1.73205) -- (-2.12132, -0.40825, 0.57735) -- (-1.41421, -1.63299, 0.57735) -- (-0.70711, -1.22474, 1.73205) -- cycle {};
\fill[facet] (-0.70711, 1.22474, 1.73205) -- (-1.41421, 0.00000, 1.73205) -- (-2.12132, -0.40825, 0.57735) -- (-2.12132, 0.40825, -0.57735) -- (-1.41421, 1.63299, -0.57735) -- (-0.70711, 2.04124, 0.57735) -- cycle {};
\fill[facet] (0.70711, 1.22474, 1.73205) -- (-0.70711, 1.22474, 1.73205) -- (-0.70711, 2.04124, 0.57735) -- (0.70711, 2.04124, 0.57735) -- cycle {};
\fill[facet] (0.70711, 2.04124, 0.57735) -- (-0.70711, 2.04124, 0.57735) -- (-1.41421, 1.63299, -0.57735) -- (-0.70711, 1.22474, -1.73205) -- (0.70711, 1.22474, -1.73205) -- (1.41421, 1.63299, -0.57735) -- cycle {};
\fill[facet] (0.70711, 1.22474, 1.73205) -- (1.41421, 0.00000, 1.73205) -- (2.12132, -0.40825, 0.57735) -- (2.12132, 0.40825, -0.57735) -- (1.41421, 1.63299, -0.57735) -- (0.70711, 2.04124, 0.57735) -- cycle {};

\draw[edge] (-1.41421, -1.63299, 0.57735) -- (-2.12132, -0.40825, 0.57735);
\draw[edge] (-1.41421, -1.63299, 0.57735) -- (-0.70711, -1.22474, 1.73205);
\draw[edge] (-2.12132, 0.40825, -0.57735) -- (-2.12132, -0.40825, 0.57735);
\draw[edge] (-2.12132, 0.40825, -0.57735) -- (-1.41421, 1.63299, -0.57735);
\draw[edge] (-2.12132, -0.40825, 0.57735) -- (-1.41421, 0.00000, 1.73205);
\draw[edge] (-0.70711, 1.22474, -1.73205) -- (-1.41421, 1.63299, -0.57735);
\draw[edge] (-0.70711, 1.22474, -1.73205) -- (0.70711, 1.22474, -1.73205);
\draw[edge] (-0.70711, -1.22474, 1.73205) -- (-1.41421, 0.00000, 1.73205);
\draw[edge] (-0.70711, -1.22474, 1.73205) -- (0.70711, -1.22474, 1.73205);
\draw[edge] (-1.41421, 1.63299, -0.57735) -- (-0.70711, 2.04124, 0.57735);
\draw[edge] (-1.41421, 0.00000, 1.73205) -- (-0.70711, 1.22474, 1.73205);
\draw[edge] (1.41421, -1.63299, 0.57735) -- (0.70711, -1.22474, 1.73205);
\draw[edge] (1.41421, -1.63299, 0.57735) -- (2.12132, -0.40825, 0.57735);
\draw[edge] (0.70711, 1.22474, -1.73205) -- (1.41421, 1.63299, -0.57735);
\draw[edge] (0.70711, -1.22474, 1.73205) -- (1.41421, 0.00000, 1.73205);
\draw[edge] (-0.70711, 2.04124, 0.57735) -- (-0.70711, 1.22474, 1.73205);
\draw[edge] (-0.70711, 2.04124, 0.57735) -- (0.70711, 2.04124, 0.57735);
\draw[edge] (-0.70711, 1.22474, 1.73205) -- (0.70711, 1.22474, 1.73205);
\draw[edge] (2.12132, 0.40825, -0.57735) -- (2.12132, -0.40825, 0.57735);
\draw[edge] (2.12132, 0.40825, -0.57735) -- (1.41421, 1.63299, -0.57735);
\draw[edge] (2.12132, -0.40825, 0.57735) -- (1.41421, 0.00000, 1.73205);
\draw[edge] (1.41421, 1.63299, -0.57735) -- (0.70711, 2.04124, 0.57735);
\draw[edge] (1.41421, 0.00000, 1.73205) -- (0.70711, 1.22474, 1.73205);
\draw[edge] (0.70711, 2.04124, 0.57735) -- (0.70711, 1.22474, 1.73205);

\end{tikzpicture}

%% file: img/maximal_face.tikz
\resizebox{!}{1.8in}{
\begin{tikzpicture}%
	[x={(0.759177cm, 0.586837cm)},
	y={(-0.456662cm, 0.788466cm)},
	z={(0.463800cm, -0.184239cm)},
	back/.style={loosely dotted, thick},
	edge/.style={color=ngreen},
	facet/.style={fill=ngreen,fill opacity=0.300000}]
\coordinate (-3.00000, -1.00000, 1.00000) at (-3.00000, -1.00000, 1.00000);
\coordinate (-3.00000, 1.00000, -1.00000) at (-3.00000, 1.00000, -1.00000);
\coordinate (-1.00000, -3.00000, 1.00000) at (-1.00000, -3.00000, 1.00000);
\coordinate (-1.00000, -1.00000, 3.00000) at (-1.00000, -1.00000, 3.00000);
\coordinate (-1.00000, 1.00000, -3.00000) at (-1.00000, 1.00000, -3.00000);
\coordinate (-1.00000, 3.00000, -1.00000) at (-1.00000, 3.00000, -1.00000);
\coordinate (1.00000, -3.00000, -1.00000) at (1.00000, -3.00000, -1.00000);
\coordinate (1.00000, -1.00000, -3.00000) at (1.00000, -1.00000, -3.00000);
\coordinate (t1) at (1.00000, 1.00000, 3.00000);
\coordinate (t2) at (1.00000, 3.00000, 1.00000);
\coordinate (3.00000, -1.00000, -1.00000) at (3.00000, -1.00000, -1.00000);
\coordinate (t3) at (3.00000, 1.00000, 1.00000);
\coordinate (arrowstart) at (1.7,1.7,1.7);
\coordinate (arrowend) at (2.7,2.7,2.7);
\draw[edge,back] (-3.00000, 1.00000, -1.00000) -- (-1.00000, 1.00000, -3.00000);
\draw[edge,back] (-1.00000, 1.00000, -3.00000) -- (-1.00000, 3.00000, -1.00000);
\draw[edge,back] (-1.00000, 1.00000, -3.00000) -- (1.00000, -1.00000, -3.00000);
\draw[edge,back] (1.00000, -3.00000, -1.00000) -- (1.00000, -1.00000, -3.00000);
\draw[edge,back] (1.00000, -1.00000, -3.00000) -- (3.00000, -1.00000, -1.00000);
\fill[facet] (t3) -- (t1) -- (-1.00000, -1.00000, 3.00000) -- (-1.00000, -3.00000, 1.00000) -- (1.00000, -3.00000, -1.00000) -- (3.00000, -1.00000, -1.00000) -- cycle {};
\fill[facet] (-1.00000, -1.00000, 3.00000) -- (-3.00000, -1.00000, 1.00000) -- (-1.00000, -3.00000, 1.00000) -- cycle {};
\fill[facet] (t2) -- (-1.00000, 3.00000, -1.00000) -- (-3.00000, 1.00000, -1.00000) -- (-3.00000, -1.00000, 1.00000) -- (-1.00000, -1.00000, 3.00000) -- (t1) -- cycle {};
\fill[facet,dblue] (t3) -- (t1) -- (t2) -- cycle {};
\draw[edge] (-3.00000, -1.00000, 1.00000) -- (-3.00000, 1.00000, -1.00000);
\draw[edge] (-3.00000, -1.00000, 1.00000) -- (-1.00000, -3.00000, 1.00000);
\draw[edge] (-3.00000, -1.00000, 1.00000) -- (-1.00000, -1.00000, 3.00000);
\draw[edge] (-3.00000, 1.00000, -1.00000) -- (-1.00000, 3.00000, -1.00000);
\draw[edge] (-1.00000, -3.00000, 1.00000) -- (-1.00000, -1.00000, 3.00000);
\draw[edge] (-1.00000, -3.00000, 1.00000) -- (1.00000, -3.00000, -1.00000);
\draw[edge] (-1.00000, -1.00000, 3.00000) -- (t1);
\draw[edge] (-1.00000, 3.00000, -1.00000) -- (t2);
\draw[edge] (1.00000, -3.00000, -1.00000) -- (3.00000, -1.00000, -1.00000);
\draw[edge] (3.00000, -1.00000, -1.00000) -- (t3);
\draw[edge,color=dblue] (t1) -- (t2);
\draw[edge,color=dblue] (t1) -- (t3);
\draw[edge,color=dblue] (t2) -- (t3);
\draw[color=norange,decoration={markings,mark=at position 1 with {\arrow[scale=2,>=stealth]{>}}},postaction={decorate}]
	(arrowstart) -- (arrowend);
\node[color=norange,anchor=north west] at (arrowend) {\hspace{10pt}\Large$(1,0,1,1)$};
\end{tikzpicture}
}

%% file: img/orbit_polytopes_3.tikz
\begin{tikzpicture}
\coordinate (point) at (0,0);
\draw[color=white,fill=white] (point) circle (.5in);
\draw[color=nred,fill=nred] (point) circle (2pt);
\node[color=nred] at (0,-.5in) {$\mathcal{O}_{(3)}$};
\end{tikzpicture}
\qquad
\begin{tikzpicture}%
	[
	edge/.style={color=dblue},
	facet/.style={fill=dblue,fill opacity=0.300000}]
%
\fill[facet] (-0.86603, 1) -- (0.00000, -.5) -- (0.86603, 1) -- cycle {};
\draw[edge] (0.00000, -.5) -- (0.86603,1);
\draw[edge] (0.00000, -.5) -- (-0.86603, 1);
\draw[edge] (0.86603, 1) -- (-0.86603, 1);
\node[color=dblue] at (0,-.5in) {$\mathcal{O}_{(2,1)}$};
\end{tikzpicture}
\qquad
\begin{tikzpicture}%
	[
	edge/.style={color=ngreen},
	facet/.style={fill=ngreen,fill opacity=0.300000}]
%
%
\coordinate (0.00000, 1.00000) at (0.00000, 1.00000);
\coordinate (0.86603, -0.50000) at (0.86603, -0.50000);
\coordinate (-0.86603, -0.50000) at (-0.86603, -0.50000);
\fill[facet] (-0.86603, -0.50000) -- (0.00000, 1.00000) -- (0.86603, -0.50000) -- cycle {};
\draw[edge] (0.00000, 1.00000) -- (0.86603, -0.50000);
\draw[edge] (0.00000, 1.00000) -- (-0.86603, -0.50000);
\draw[edge] (0.86603, -0.50000) -- (-0.86603, -0.50000);
\node[color=ngreen] at (0,-.5in) {$\mathcal{O}_{(1,2)}$};
\end{tikzpicture}
\qquad
\begin{tikzpicture}%
	[
	edge/.style={color=norange},
	facet/.style={fill=norange,fill opacity=0.300000}]
%
%
\coordinate (0.86603, -0.50000) at (0.86603, -0.50000);
\coordinate (0.86603, 0.50000) at (0.86603, 0.50000);
\coordinate (0.00000, -1.00000) at (0.00000, -1.00000);
\coordinate (0.00000, 1.00000) at (0.00000, 1.00000);
\coordinate (-0.86603, -0.50000) at (-0.86603, -0.50000);
\coordinate (-0.86603, 0.50000) at (-0.86603, 0.50000);
\fill[facet] (-0.86603, 0.50000) -- (0.00000, 1.00000) -- (0.86603, 0.50000) -- (0.86603, -0.50000) -- (0.00000, -1.00000) -- (-0.86603, -0.50000) -- cycle {};
\draw[edge] (0.86603, -0.50000) -- (0.86603, 0.50000);
\draw[edge] (0.86603, -0.50000) -- (0.00000, -1.00000);
\draw[edge] (0.86603, 0.50000) -- (0.00000, 1.00000);
\draw[edge] (0.00000, -1.00000) -- (-0.86603, -0.50000);
\draw[edge] (0.00000, 1.00000) -- (-0.86603, 0.50000);
\draw[edge] (-0.86603, -0.50000) -- (-0.86603, 0.50000);
\node[color=norange] at (0,-.5in) {$\mathcal{O}_{(1,1,1)}$};
\end{tikzpicture}